\documentclass[12pt,reqno]{article}
\usepackage{amssymb,amscd,amsmath,amsthm,color}
\newtheorem{thm}{Theorem}[section]
\newtheorem{prop}[thm]{Proposition}
\newtheorem{lemma}[thm]{Lemma}
\newtheorem{cor}[thm]{Corollary}

\newtheorem{remark}[thm]{Remark}
\newtheorem{example}[thm]{Example}

\numberwithin{equation}{section}

\def\bM{\mathbb{M}}
\def\bP{\mathbb{P}}
\def\bN{\mathbb{N}}
\def\bR{\mathbb{R}}
\def\bC{\mathbb{C}}
\def\Tr{\mathrm{Tr}\,}
\def\Im{\mathrm{Im}\,}
\def\id{\mathrm{id}}
\def\ffi{\varphi}
\def\cI{\mathcal{I}}
\def\eps{\varepsilon}

\begin{document}
\baselineskip=16pt
\allowdisplaybreaks

\ \vskip 1cm 
\centerline{\LARGE Concavity of certain matrix trace and norm functions}
\bigskip
\bigskip
\centerline{\Large
Fumio Hiai\footnote{E-mail: hiai.fumio@gmail.com}}

\medskip
\begin{center}
$^1$\,Tohoku University (Emeritus), \\
Hakusan 3-8-16-303, Abiko 270-1154, Japan
\end{center}

\medskip
\begin{abstract}
We refine Epstein's method to prove joint concavity/convexity of matrix
trace functions of the extended Lieb type
$\Tr\bigl\{\Phi(A^p)^{1/2}\Psi(B^q)\Phi(A^p)^{1/2}\bigr\}^s$, where $\Phi$ and $\Psi$
are positive linear maps. By the same method combined with majorization technique,
similar properties are proved for symmetric (anti-) norm functions of the form
$\|\{\Phi(A^p)\,\sigma\,\Psi(B^q)\}^s\|$ involving an operator mean $\sigma$.
Carlen and Lieb's variational method is also used to improve the convexity property of
norm functions $\|\Phi(A^p)^s\|$.

\bigskip\noindent
{\it 2010 Mathematics Subject Classification:}
Primary 15A60, 47A30, 47A60

\bigskip\noindent
{\it Key words and phrases:}
Matrices, trace, symmetric norms, symmetric anti-norms, joint concavity, joint convexity,
operator mean
\end{abstract}

\section*{Introduction}

In the present paper we consider matrix functions of the following types:
\begin{itemize}
\item[(i)] $F(A,B)=\bigl\{\Phi(A^p)^{1/2}\Psi(B^q)\Phi(A^p)^{1/2}\bigr\}^s$,
\item[(ii)] $F(A,B)=\{\Phi(A^p)\,\sigma\,\Psi(B^q)\}^s$,
\item[(iii)] $F(A)=\Phi(A^p)^s$.
\end{itemize}
Here, the variables $A$ and $B$ are positive definite matrices, $\Phi$ and $\Psi$ are
(strictly) positive linear maps between matrix algebras, and $p,q,s$ are real parameters.
Furthermore, $\sigma$ in (ii) is an operator mean in the Kubo-Ando sense \cite{KA}. For
matrix functions $F$ in the above we are mostly interested in the range of
the parameters $p,q,s$ (or $p,s$) for which the function
$$
(A,B)\longmapsto\|F(A,B)\|\quad(\mbox{or}\ \ A\longmapsto\|F(A)\|)
$$
is convex for any symmetric norm $\|\cdot\|$ (and for every $\Phi,\Psi$), and also
for which the function $(A,B)\mapsto\|F(A,B)\|_!$ (or $A\mapsto\|F(A)\|_!$) is concave
for any symmetric anti-norm $\|\cdot\|_!$ (and for every $\Phi,\Psi$). Here, the notion of
symmetric anti-norms was recently introduced in \cite{BH1} while that of symmetric norms
is familiar in matrix analysis (see, e.g., \cite{Bh,Hi2}). A symmetric anti-norm is
a non-negative functional on the positive part of a matrix algebra that is positively
homogeneous, invariant under unitary conjugation and superadditive (opposite to
subadditivity of symmetric norms). Therefore, it is a concave functional, that is the
reason why we take a symmetric anti-norm for concavity assertions while a symmetric norm
is for convexity assertions. It is worth noting that the trace functional is a symmetric
norm and a symmetric anti-norm in common.

For instance, when $s=1$, $\Phi(A):=X^*AX$ and $\Psi=\id$, the function (i) under the
trace is
$$
(A,B)\longmapsto\Tr X^*A^pXB^q,
$$
whose joint concavity/convexity is famous as Lieb's concavity/convexity \cite{Li}. It is
well-known that an equivalent reformulation is matrix concavity/convexity of
$(A,B)\mapsto A^p\otimes B^q$ due to Ando \cite{An}. When $0<p\le1$, $s=1/p$ and
$\Phi(A):=X^*AX$, the function (iii) under the trace is
$$
A\longmapsto\Tr(X^*A^pX)^{1/p},
$$
whose concavity was first proved by Epstein \cite{Ep} by a powerful method
using theory of Pick functions (often called Epstein's method). The method was
applied in \cite{Hi1} to prove (joint) concavity of trace functions of types
(i)--(iii) under certain respective conditions on $p,q,s$. The Minkowski type trace
function (or the trace function for the matrix power means) $\Tr(A^p+B^p)^{1/p}$ was
discussed in \cite{CL1} (also \cite{AnHi,Be}), which is a special case of the function
(ii) under the trace where $s=1/p$, $\Phi=\Psi=\id$ and $\sigma$ is the arithmetic mean.
More recently in \cite{CL2}, Carlen and Lieb extensively developed concavity/convexity
properties of the trace functions of the forms $\Tr(X^*A^pX)^s$ and $\Tr(A^p+B^p)^s$.
The most remarkable in \cite{CL2} is the new method using a variational
expression for $\Tr(X^*A^pX)^s$.
Furthermore in \cite{JR}, Jen\v cov\'a and Ruskai obtained equality conditions for Lieb's
concavity/convexity as well as for some related inequalities. In this way, the functions
of the above types (i)--(iii) cover many of important cases appearing in the study of
concavity/convexity of various matrix trace functions so far.

The present paper is a continuation of \cite{Hi1}. Our strategy here is two-fold. We first
refine Epstein's method to extend some known concavity/convexity results for trace
functions as much as possible. After this is done we further extend the results with the
trace to those with symmetric (anti-) norms by using the majorization method. In Section 1
we treat the function (i) under the trace and prove its joint concavity/convexity under
suitable conditions on $p,q,s$ by using Epstein's method. In Section 2 we prove joint
concavity/convexity of the function (ii) with symmetric (anti-) norms under suitable
conditions on $p,q,s$. In Section 3, by specializing the function (ii) and
also applying the variational method of Carlen and Lieb \cite{CL2}, we obtain
concavity/covexity results for the function (iii) with symmetric (anti-) norms.
In Section 4 we examine necessary conditions on the parameters $p,q,s$ (or $p,s$) for
concavity/convexity of the relevant functions, and compare them with sufficient conditions
obtained in Sections 1--3. Here, Bekjan's idea in \cite{Be} (also used in \cite{CL2}) is
of particular use. Since it does not seem easy to extend the result of Section 1 to
functions with symmetric (anti-) norms, we consider, in Section 5, the function (i) with
the operator norm and the smallest singular value, which are particular cases of the
Ky Fan (anti-) norms. Note that convexity under all symmetric norms is reduced to that
under all Ky Fan norms and concavity under all symmetric anti-norms is to that under all
Ky Fan anti-norms.

\section{Trace functions of Lieb type}

We begin with fixing some common notations. For each $n\in\bN$ the $n\times n$ complex
matrix algebra is denoted by $\bM_n$. We write $\bM_n^+:=\{A\in\bM_n:A\ge0\}$, the
$n\times n$ positive semidefinite matrices, and $\bP_n:=\{A\in\bM_n:A>0\}$, the
$n\times n$ positive definite matrices. The usual trace on $\bM_n$ is denoted by $\Tr$.
A linear map $\Phi:\bM_n\to\bM_m$ is positive if $A\in\bM_n^+$ implies $\Phi(A)\in\bM_m^+$,
and it is strictly positive if $A\in\bP_n$ implies $\Phi(A)\in\bP_m$. Clearly, a positive
linear map $\Phi:\bM_n\to\bM_m$ is strictly positive if $\Phi(I_n)\in\bP_m$, where $I_n$
is the identity of $\bM_n$.

In this section we consider joint concavity and convexity of the trace function
\begin{equation}\label{F-1.1}
(A,B)\in\bP_n\times\bP_m\longmapsto
\Tr\bigl\{\Phi(A^p)^{1/2}\Psi(B^q)\Phi(A^p)^{1/2}\bigr\}^s,
\end{equation}
where $\Phi$ and $\Psi$ are (strictly) positive linear maps between matrix algebras. In
particular, when $s=1$, $\Phi(A):=X^*AX$ and $\Psi=\id$, the identity map, the above
function is
\begin{equation}\label{F-1.2}
(A,B)\in\bP_n\times\bP_n\longmapsto\Tr X^*A^pXB^q,
\end{equation}
for which Lieb's concavity (also convexity) is well-known \cite{Li} (also \cite{An}).

Throughout the section we assume that $(p,q)\ne(0,0)$ and $s\ne0$; otherwise, the function
\eqref{F-1.1} is constant. The next theorem extends \cite[Theorem 2.3]{Hi1}.

\begin{thm}\label{T-1.1}
Let $n,m,l\in\bN$. Let $\Phi:\bM_n\to\bM_l$ and $\Psi:\bM_m\to\bM_l$ be strictly positive
linear maps. 
\begin{itemize}
\item[\rm(1)] If either $0\le p,q\le1$ and $1/2\le s\le1/(p+q)$, or $-1\le p,q\le0$ and
$1/(p+q)\le s\le-1/2$, then the function \eqref{F-1.1} is jointly concave.
\item[\rm(2)] If either $0\le p,q\le1$ and $-1/(p+q)\le s\le-1/2$, or $-1\le p,q\le0$ and $1/2\le s\le-1/(p+q)$, then the function \eqref{F-1.1} is jointly convex.
\end{itemize}
\end{thm}

\begin{proof}
(1)\enspace
First, assume that $0\le p,q\le1$ and $1/2\le s\le1/(p+q)$. Although the proof below is
a slight improvement of Epstein's method in the proof of \cite[Theorem 2.3]{Hi1}, we shall
present it in detail while the proofs for other cases in (1) and (2) will only be sketched.
Let us first show that the assertion in the case $1/2\le s<1/(p+q)$ follows from that in
the case $s=1/(p+q)$. Indeed, when $1/2\le s<1/(p+q)$ (and also $0<p,q\le1$), one can
choose $p'\in[p,1]$ and $q'\in[q,1]$ such that $s=1/(p'+q')$. Let $A_1,A_2\in\bM_n^+$,
$B_1,B_2\in\bM_m^+$, and $0<\lambda<1$. Then, since $x^\alpha$ ($x\ge0$) with
$0<\alpha\le1$ is operator concave as well as operator monotone, we have
$$
(\lambda A_1+(1-\lambda)A_2)^p
\ge(\lambda A_1^{p/p'}+(1-\lambda)A_2^{p/p'})^{p'}
$$
so that
$$
\Phi((\lambda A_1+(1-\lambda)A_2)^p)
\ge\Phi((\lambda A_1^{p/p'}+(1-\lambda)A_2^{p/p'})^{p'}),
$$
and similarly
$$
\Psi((\lambda B_1+(1-\lambda)B_2)^q)
\ge\Psi((\lambda B_1^{q/q'}+(1-\lambda)B_2^{q/q'})^{q'}).
$$
From the assertion in the case $s=1/(p'+q')$ we have
\begin{align*}
&\Tr\big\{\Phi((\lambda A_1+(1-\lambda)A_2)^p)^{1/2}
\Psi((\lambda B_1+(1-\lambda)B_2)^q)
\Phi((\lambda A_1+(1-\lambda)A_2)^p)^{1/2}\bigr\}^s \\
&\qquad\ge\Tr\big\{\Phi((\lambda A_1^{p/p'}+(1-\lambda)A_2^{p/p'})^{p'})^{1/2}
\Psi((\lambda B_1^{q/q'}+(1-\lambda)B_2^{q/q'})^{q'}) \\
&\qquad\qquad\qquad\qquad\times
\Phi((\lambda A_1^{p/p'}+(1-\lambda)A_2^{p/p'})^{p'})^{1/2}\bigr\}^s \\
&\qquad\ge\lambda\Tr\bigl\{\Phi((A_1^{p/p'})^{p'})^{1/2}
\Psi((B_1^{q/q'})^{q'})\Phi((A_1^{p/p'})^{p'})^{1/2}\bigr\}^s \\
&\qquad\qquad\qquad+(1-\lambda)\Tr\bigl\{\Phi((A_2^{p/p'})^{p'})^{1/2}
\Psi((B_2^{q/q'})^{q'})\Phi((A_2^{p/p'})^{p'})^{1/2}\bigr\}^s \\
&\qquad=\lambda\Tr\bigl\{\Phi(A_1)^{1/2}\Psi(B_1)\Phi(A_1)^{1/2}\bigr\}^s
+(1-\lambda)\Tr\bigl\{\Phi(A_2^p)^{1/2}\Psi(B_2^q)\Phi(A_2^p)^{1/2}\bigr\}^s.
\end{align*}
Therefore, in the following proof we may assume that $s=1/(p+q)$.

Set $\gamma:=p+q\in(0,2]$ and so $s=1/\gamma$. As in \cite{Hi1} we will use the following
notations:
$$
\bC^+:=\{z\in\bC:\Im z>0\},
$$
$$
\cI_n^+:=\{X\in\bM_n:\Im X>0\},\qquad\cI_n^-:=\{X\in\bM_n:\Im X<0\},
$$
and
$$
\Gamma_{\gamma\pi}:=\{re^{i\theta}:r>0,\,0<\theta<\gamma\pi\}.
$$
Note that for each $\alpha>0$ the function $x^\alpha$ ($x>0$) has the analytic continuation
$z^\alpha$ in $\bC\setminus[0,\infty)$ (in particular, in $\bC^+$) defined by
$$
z^\alpha:=r^\alpha e^{i\alpha\theta}
\quad\mbox{for}\quad z=re^{i\theta}\ \ (r>0,\ 0<\theta<2\pi).
$$
To obtain the joint concavity result, it suffices to prove that if $A,H\in\bM_n$ and
$B,K\in\bM_m$ are such
that $A,B>0$ and $H,K$ are Hermitian, then
$$
{d^2\over dx^2}\Tr \bigl\{\Phi((A+xH)^p)^{1/2}
\Psi((B+xK)^q)\Phi((A+xH)^p)^{1/2}\bigr\}^s\le0
$$
for every sufficiently small $x>0$.

For $z\in\bC$ set $X(z):=zA+H$ and $Y(z):=zB+K$. For any $z\in\bC^+$, since
$X(z)\in\cI_n^+$, $Y(z)\in\cI_m^+$ and $p,q\in(0,1]$, we can define $X(z)^p$ and $Y(z)^q$
by analytic functional calculus by \cite[Lemma 1.1]{Hi1}. Since \cite[Lemma 1.2]{Hi1}
implies that
\begin{equation}\label{F-1.3}
\Im\Phi(X(z)^p)=\Phi(\Im X(z)^p)>0,\qquad\Im\Psi(Y(z)^q)=\Psi(\Im Y(z)^q)>0,
\end{equation}
we have $\Phi(X(z)^p),\Psi(Y(z)^q)\in\cI_l^+$ and hence $\Phi(X(z)^p)^{1/2}\in\cI_l^+$ is
also well-defined. Now define
\begin{equation}\label{F-1.4}
F(z):=\Phi(X(z)^p)^{1/2}\Psi(Y(z)^q)\Phi(X(z)^p)^{1/2},
\qquad z\in\bC^+,
\end{equation}
which is analytic in $\bC^+$. We will prove that
\begin{equation}\label{F-1.5}
\sigma(F(z))\subset\Gamma_{\gamma\pi}\quad{\rm if}\quad z\in\bC^+,
\end{equation}
where $\sigma(F(z))$ is the set of the eigenvalues of $F(z)$. To prove this, it suffices to
show the following properties:
\begin{itemize}
\item[(a)] When $z=re^{i\theta}$ with a fixed $0<\theta<\pi$,
$\sigma(F(z))\subset\Gamma_{\gamma\pi}$ for sufficiently large $r>0$.
\item[(b)] $\sigma(F(z))\cap[0,\infty)=\emptyset$ for all $z\in\bC^+$.
\item[(c)] $\sigma(F(z))\cap\{re^{i\gamma\pi}:r\ge0\}=\emptyset$ for all $z\in\bC^+$.
\end{itemize}
In fact, if \eqref{F-1.5} fails to hold for some $z_0=r_0e^{i\theta_0}\in\bC^+$, then
according to (a) and the continuity of the eigenvalues of $F(z)$ we must have
$\sigma(F(z))\cup\partial\Gamma_{\gamma\pi}\ne\emptyset$ for some
$z\in\{re^{\theta_0}:r>r_0\}$, which means that (b) or (c) must be violated.

{\it Proof of} (a).\enspace Choose an $R>0$ such that $A>R^{-1}\|H\|I_n$ and
$B>R^{-1}\|K\|I_m$. Then one can define an analytic function
$$
\tilde F(z):=z^\gamma\Phi((A+z^{-1}H)^p)^{1/2}
\Psi((B+z^{-1}K)^q)\Phi((A+z^{-1}H)^p)^{1/2}
$$
in $z\in\bC^+$ with $|z|>R$ (see \cite[Section 1]{Hi1}). It is easy to see that
$\tilde F(z)$ and $F(z)$ are continuously extended to the interval $(R,\infty)$ of the
real line so that
\begin{align*}
\tilde F(x)&=x^\gamma\Phi((A+x^{-1}H)^p)^{1/2}
\Psi((B+x^{-1}K)^q)\Phi((A+x^{-1}H)^p)^{1/2} \\
&=\Phi((xA+H)^p)^{1/2}\Psi((xB+K)^q)\Phi((xA+H)^p)^{1/2}=F(x),\quad x\in(R,\infty).
\end{align*}
We thus obtain for every $z\in\bC^+$ with $|z|>R$
\begin{equation}\label{F-1.6}
F(z):=z^\gamma\Phi((A+z^{-1}H)^p)^{1/2}
\Psi((B+z^{-1}K)^q)\Phi((A+z^{-1}H)^p)^{1/2}.
\end{equation}
When $z=re^{i\theta_0}$ with $0<\theta_0<\pi$ fixed and $r\to\infty$, note that
$$
\sigma\bigl(\Phi((A+z^{-1}H)^p)^{1/2}
\Psi((B+z^{-1}K)^q)\Phi((A+z^{-1}H)^p)^{1/2}\bigr)
$$
converges to
$S:=\sigma\bigl(\Phi(A^p)^{1/2}\Psi(B^q)\Phi(A^p)^{1/2}\bigr)\subset(0,\infty)$. Since
$(r^\gamma e^{i\gamma\theta_0})^{-1}\sigma(F(re^{i\theta_0}))$ converges to $S$ as
$r\to\infty$, we see that (a) holds.

{\it Proof of} (b).\enspace For any $r\in[0,\infty)$ we have
$$
F(z)-rI_l=\Phi(X(z)^p)^{1/2}\bigl\{\Psi(Y(z)^q)-r\Phi(X(z)^p)^{-1}\bigr\}
\Phi(X(z)^p)^{1/2}.
$$
Since $\Phi(X(z)^p),\Psi(Y(z)^q)\in\cI_l^+$ as already mentioned, we see by
\cite[Lemma 1.1]{Hi1} that
$$
\Psi(Y(z)^q)-r\Phi(Y(z)^p)^{-1}\in\cI_l^+
$$
so that $F(z)-rI_l$ is invertible.

{\it Proof of} (c).\enspace For any $r\in[0,\infty)$ we have
$$
F(z)-re^{i\gamma\pi}I_l=e^{iq\pi}\Phi(X(z)^p)^{1/2}\bigl\{\Psi(e^{-iq\pi}Y(z)^q)
-r\Phi(e^{-ip\pi}X(z)^p)^{-1}\bigr\}\Phi(X(z)^p)^{1/2}
$$
thanks to $p+q=\gamma$. Since
$\Psi(e^{-iq\pi}Y(z)^q)-r\Phi(e^{-ip\pi}X(z)^p)^{-1}\in\cI_l^-$
by \cite[Lemma 1.2]{Hi1}, $F(z)-re^{i\gamma\pi}I_l$ is invertible.

We have shown \eqref{F-1.5}. Hence we can define $F(z)^s$ for $z\in\bC^+$ by applying the
analytic functional calculus by $z^s$ on $\Gamma_{\gamma\pi}$ to $F(z)$. Since
$\gamma s=1$ by assumption, note that $z^s$ maps $\Gamma_{\gamma\pi}$ into $\bC^+$. Thus,
$F(z)^s$ is an analytic function such that $\sigma(F(z)^s)\subset\bC^+$ and so
$\Tr(F(z)^s)\in\bC^+$ for all $z\in\bC^+$ (see \cite[Section 1]{Hi1}). In view of
\eqref{F-1.6}, $F(z)^s$ in $\bC^+$ is continuously extended to the interval $(R,\infty)$
so that
$$
F(x)^s=x\bigl\{\Phi((A+x^{-1}H)^p)^{1/2}\Psi((B+x^{-1}K)^q)
\Phi((A+x^{-1}H)^p)^{1/2}\bigr\}^s,\quad x\in(R,\infty).
$$
Since $\Tr(F(x)^s)\in\bR$ for all $x\in(R,\infty)$, by the reflection principle we obtain
a Pick function $\ffi$ on $\bC\setminus(-\infty,R]$ such that $\ffi(x)=\Tr(F(x)^s)$ for all
$x\in(R,\infty)$. For every $x\in(0,R^{-1})$ we have
\begin{equation}\label{F-1.7}
x\ffi(x^{-1})=\Tr\bigl\{\Phi((A+xH)^p)^{1/2}
\Psi((B+xK)^q)\Phi((A+xH)^p)^{1/2}\bigr\}^s.
\end{equation}
It thus remains to show that
\begin{equation}\label{F-1.8}
{d^2\over dx^2}(x\ffi(x^{-1}))\le0,\qquad x\in(0,R^{-1}).
\end{equation}

According to Nevanlinna's theorem for Pick functions (see, e.g.,
\cite[Theorem 2.6.2]{Hi2}), $\ffi$ admits an integral expression
\begin{equation}\label{F-1.9}
\ffi(z)=a+bz+\int_{-\infty}^\infty{1+tz\over t-z}\,d\nu(t),
\end{equation}
where $a\in\bR$, $b\ge0$, and $\nu$ is a finite measure on $\bR$. Since $\ffi$ is
analytically continued across the interval $(R,\infty)$, the measure $\nu$ is supported in
$(-\infty,R]$. Therefore,
$$
x\ffi(x^{-1})=ax+b+\int_{-\infty}^R{x(x+t)\over xt-1}\,d\nu(t),\qquad x\in(0,R^{-1}).
$$
Compute
$$
{d\over dx}\biggl({x(x+t)\over xt-1}\biggr)={x^2t-2x-t\over(xt-1)^2},\qquad
{d^2\over dx^2}\biggl({x(x+t)\over xt-1}\biggr)={2(t^2+1)\over(xt-1)^3}<0
$$
for all $x\in(0,R^{-1})$ and all $t\in(-\infty,R]$, and hence \eqref{F-1.8} follows.

The proof for the second case where $-1\le p,q\le0$ and $1/(p+q)\le s\le-1/2$ can be done
similarly to the above but a more convenient way is to replace $\Phi$ and $\Psi$ with
$\hat\Phi(A):=\Phi(A^{-1})^{-1}$ for $A\in\bP_n$ and $\hat\Psi(B):=\Psi(B^{-1})^{-1}$
for $B\in\bP_m$, respectively. Although $\hat\Phi$ and $\hat\Psi$ are no longer linear,
the above proof can work with $\hat\Phi$ and $\hat\Psi$ in place of $\Phi$ and $\Psi$.
Indeed, in the above we only used monotonicity, property \eqref{F-1.3} and
positive homogeneity for $\Phi,\Psi$, which are valid for $\hat\Phi,\hat\Psi$ too. Since
\begin{equation}\label{F-1.10}
\bigl\{\hat\Phi(A^p)^{1/2}\hat\Psi(B^q)\hat\Phi(A^p)^{1/2}\bigr\}^s
=\bigl\{\Phi(A^{-p})^{1/2}\Psi(B^{-q})\Phi(A^{-p})^{1/2}\bigr\}^{-s},
\end{equation}
the second case of (1) immediately follows from the first case for $\hat\Phi$ and
$\hat\Psi$.

(2)\enspace
To prove the first case of (2), assume that $0\le p,q\le1$ and $-1/(p+q)\le s\le-1/2$.
As in the proof of (1) we may assume that $s=-1/(p+q)$. Then the proof is similar to
the above (1). In the present case, the analytic function $F(z)^s$ in $\bC^+$ satisfies
$\sigma(F(z)^s)\subset\bC^-$ for all $z\in\bC^+$ and is continuously extended to
$(R,\infty)$ so that
$$
F(x)^s=x^{-1}\bigl\{\Phi((A+x^{-1}H)^p)^{1/2}\Psi((B+x^{-1}K)^q)
\Phi((A+x^{-1}H)^p)^{1/2}\bigr\}^s,\quad x\in(R,\infty).
$$
Thus, by considering $\bigl\{\Tr(F(z)^s)\bigr\}^{-1}$ in place of $\Tr(F(z)^s)$, we obtain
a Pick function $\ffi$ on $\bC\setminus(-\infty,R]$ such that
$\ffi(x)=\bigl\{\Tr(F(x)^s)\}^{-1}$ for all $x\in(R,\infty)$. Since
$$
x\ffi(x^{-1})=\bigl(\Tr\bigl\{\Phi((A+xH)^p)^{1/2}
\Psi((B+xK)^q)\Phi((A+xH)^p)^{1/2}\bigr\}^s\bigr)^{-1},\quad x\in(0,R^{-1}),
$$
the function $\bigl(\Tr\bigl\{\Phi(A^p)^{1/2}\Psi(B^q)\Phi(A^p)^{1/2}\bigr\}\bigr)^{-1}$
is jointly concave in $(A,B)\in\bP_n\times\bP_m$ in the same way as above, which implies
that \eqref{F-1.1} is jointly convex. The argument in the last paragraph of the proof
of (1) can also work to prove the second case of (2).
\end{proof}

It is obvious by convergence that in Theorem \ref{T-1.1} strict positivity of $\Phi,\Psi$
is relaxed to the usual positivity and $A,B>0$ is to $A,B\ge0$ as far as all the parameters
$p$, $q$ and $s$ are non-negative. (Here, for $A\in\bM_n^+$ both conventions of $A^0$
being $I_n$ and of $A^0$ being the support projection of $A$ are available.) This remark
will be available throughout the paper.

\section{Norm functions involving operator means}

Before going into the main topic of this section we recall symmetric anti-norms
introduced in \cite{BH1}. A norm $\|\cdot\|$ on $\bM_n$ is said to be {\it symmetric} or
{\it unitarily invariant} if $\|UXV\|=\|X\|$ for all $X\in\bM_n$ and unitaries
$U,V\in\bM_n$. On the other hand, a {\it symmetric anti-norm} $\|\cdot\|_!$ on $\bM_n^+$
is a non-negative continuous functional such that
\begin{itemize}
\item[(a)] $\|\lambda A\|_!=\lambda\|A\|_!$ for all $A\in \bM_n^+ $ and all reals
$\lambda\ge0$,
\item[(b)] $\|A\|_!=\|UAU^*\|_!$ for all $A\in \bM_n^+ $ and all unitaries $U$,
\item[(c)] $\|A+B\|_!\ge\|A\|_!+\|B\|_!$ for all $A,\,B\in\bM_n^+$.
\end{itemize}
We note that a symmetric norm $\|\cdot\|$, when restricted on $\bM_n^+$, is also
characterized by the same (a), (b) and the double inequality
$\|A\|\le\|A+B\|\le\|A\|+\|B\|$ in place of the superadditivity axiom in (c). So, the
notion of symmetric anti-norms is a natural superadditive counterpart of that of symmetric
norms. We have quite a few examples of symmetric anti-norms on $\bM_n^+$. The following
are among important examples \cite{BH1,BH2}.

\begin{example}\label{E-2.1}\rm
We write $\lambda_j^\uparrow(A)$, $j=1,\dots,n$, for the eigenvalues of $A\in\bM_n^+$ in
increasing order with counting multiplicities, and similarly $\lambda_j^\downarrow(A)$,
$j=1,\dots,n$, for the eigenvalues of $A$ in decreasing order.
\begin{itemize}
\item[(i)] For $k=1,\dots,n$ the {\it Ky Fan $k$-anti-norm} is
$$
\|A\|_{\{k\}}:=\sum_{j=1}^k\lambda_j^\uparrow(A).
$$
This is the anti-norm version of the {\it Ky Fan $k$-norm}
$\|A\|_{(k)}:=\sum_{j=1}^k\lambda_j^\downarrow(A)$. It is remarkable that the trace
functional $\Tr A=\|A\|_{(n)}=\|A\|_{\{n\}}$ is a symmetric norm and a symmetric anti-norm
simultaneously.
\item[(ii)] The {\it Schatten quasi-norm} $\|A\|_p:=\{\Tr(A^p)\}^{1/p}$ when $0<p<1$ is a
symmetric anti-norm. For $p>0$ so is the {\it negative Schatten anti-norm}
$\|A\|_{-p}:=\{\Tr(A^{-p})\}^{-1/p}$ (defined to be $0$ unless $A$ is invertible). 
\item[(iii)] For $k=1,\dots,n$ the functional of Minkowski type
$$
\Delta_k(A):=\Biggl\{\prod_{j=1}^k\lambda_j^\uparrow(A)\Biggr\}^{1/k}.
$$
is a symmetric anti-norm. In particular, $\Delta_n(A)=\det^{1/n}A$ is the so-called
{\it Minkowski functional}.
\end{itemize}
\end{example}

In this section we deal with convexity or concavity properties for symmetric norm or
anti-norm functions of the form $\|\{\Phi(A^p)\,\sigma\,\Psi(B^q)\}^s\|$ or
$\|\{\Phi(A^p)\,\sigma\,\Psi(B^q)\}^s\|_!$ involving an operator mean $\sigma$ in the
Kubo-Ando sense \cite{KA}. As in the previous section we assume that $(p,q)\ne(0,0)$ and
$s\ne0$.

\begin{thm}\label{T-2.2}
Let $n,m,l\in\bN$. Let $\Phi:\bM_n\to\bM_l$ and $\Psi:\bM_m\to\bM_l$ be strictly positive
linear maps, and $\sigma$ be any operator mean in the Kubo-Ando sense. Assume that either
$0\le p,q\le1$ and $0<s\le1/\max\{p,q\}$, or $-1\le p,q\le0$ and $1/\min\{p,q\}\le s<0$.
\begin{itemize}
\item[\rm(1)] For every symmetric anti-norm $\|\cdot\|_!$ on $\bM_l^+$ the function
$$
(A,B)\in\bP_n\times\bP_m\longmapsto\|\{\Phi(A^p)\,\sigma\,\Psi(B^q)\}^s\|_!
$$
is jointly concave.
\item[\rm(2)] For every symmetric norm $\|\cdot\|$ on $\bM_l$ the function
$$
(A,B)\in\bP_n\times\bP_m\longmapsto\|\{\Phi(A^p)\,\sigma\,\Psi(B^q)\}^{-s}\|^{-1}
$$
is jointly concave, and hence the function
$$
(A,B)\in\bP_n\times\bP_m\longmapsto\|\{\Phi(A^p)\,\sigma\,\Psi(B^q)\}^{-s}\|
$$
is jointly convex.
\end{itemize}
\end{thm}

To prove joint concavity of the anti-norm function, we shall use the particular case with
the trace function, which we first show as a lemma. Even this trace function case extends
\cite[Theorem 4.3]{Hi1}.

\begin{lemma}\label{L-2.3}
Let $\Phi$, $\Psi$ and $\sigma$ be as in Theorem \ref{T-2.2}. Under the same assumptions
of $p$, $q$ and $s$ as in Theorem \ref{T-2.2}, the function
$$
(A,B)\in\bP_n\times\bP_m\longmapsto
\Tr\{\Phi(A^p)\,\sigma\,\Psi(B^q)\}^s
$$
is jointly concave.
\end{lemma}

\begin{proof}
First, assume that $0<p,q\le1$ and $0<s\le1$. Let $A_1,A_2\in\bP_n$ and $B_1,B_2\in\bP_m$. 
By monotonicity and joint concavity of $\sigma$ we have
\begin{align}
&\Phi\biggl(\biggl({A_1+A_2\over2}\biggr)^p\biggr)
\,\sigma\,\Psi\biggl(\biggl({B_1+B_2\over2}\biggr)^q\biggr) \nonumber\\
&\qquad\ge\biggl({\Phi(A_1^p)+\Phi(A_2^p)\over2}\biggr)
\,\sigma\,\biggl({\Psi(B_1^q)+\Psi(B_2^q)\over2}\biggr) \label{F-2.1}\\
&\qquad\ge{1\over2}\bigl\{(\Phi(A_1^p)\,\sigma\,\Psi(B_1^q))
+(\Phi(A_2^p)\,\sigma\,\Psi(B_2^q))\bigr\} \nonumber
\end{align}
and so
\begin{align*}
&\biggl\{\Phi\biggl(\biggl({A_1+A_2\over2}\biggr)^p\biggr)
\,\sigma\,\Psi\biggl(\biggl({B_1+B_2\over2}\biggr)^q\biggr)\biggr\}^s \\
&\qquad\ge{1\over2}\bigl\{(\Phi(A_1^p)\,\sigma\,\Psi(B_1^q))^s
+(\Phi(A_2^p)\,\sigma\,\Psi(B_2^q))^s\bigr\}.
\end{align*}
Secondly, assume that $0<p,q\le1$ and $1\le s\le1/\max\{p,q\}$. By taking account of the
transposed mean $A\,\sigma'\,B:=B\,\sigma\,A$ \cite{KA}, we may further assume that
$q\le p$ and so $1\le s\le1/p$. We show that the assertion in this case follows from that
in the more particular case $q=p$ and $s=1/p$. Indeed, let $p':=1/s\in[p,1]$. Since
\begin{align*}
\Phi\biggl(\biggl({A_1+A_2\over2}\biggr)^p\biggr)
&\ge\Phi\biggl(\biggl({A_1^{p/p'}+A_2^{p/p'}\over2}\biggr)^{p'}\biggr), \\
\Psi\biggl(\biggl({B_1+B_2\over2}\biggr)^q\biggr)
&\ge\Psi\biggl(\biggl({B_1^{q/p'}+B_2^{q/p'}\over2}\biggr)^{p'}\biggr),
\end{align*}
the joint concavity assertion for $p,q$ both replaced with $p'$ implies that
\begin{align*}
&\Tr\biggl\{\Phi\biggl(\biggl({A_1+A_2\over2}\biggr)^p\biggr)
\,\sigma\,\Psi\biggl(\biggl({B_1+B_2\over2}\biggr)^q\biggr)\biggr\}^s \\
&\qquad\ge\Tr\biggl\{\Phi\biggl(\biggl({A_1^{p/p'}+A_2^{p/p'}\over2}\biggr)^{p'}\biggr)
\,\sigma\,\Psi\biggl(\biggl({B_1^{q/p'}+B_2^{q/p'}\over2}\biggr)^{p'}\biggr)\biggr\}^s \\
&\qquad\ge{1\over2}\Bigl(
\Tr\bigl\{\Phi((A_1^{p/p'})^{p'})\,\sigma\,\Psi((B_1^{q/p'})^{p'})\bigr\}^s
+\Tr\bigl\{\Phi((A_2^{p/p'})^{p'})\,\sigma\,\Psi((B_2^{q/p'})^{p'})\bigr\}^s\Bigr) \\
&\qquad={1\over2}\Bigl(\Tr\bigl\{\Phi(A_1^p)\,\sigma\,\Psi(B_1^q)\bigr\}^s
+\Tr\bigl\{\Phi(A_2^p)\,\sigma\,\Psi(B_2^q)\bigr\}^s\Bigr).
\end{align*}
Hence the proof is reduced to joint concavity of
$\Tr\{\Phi(A^p)\,\sigma\,\Psi(B^p)\}^{1/p}$ when $0<p\le1$. Now the proof can be done by
a slight modification of that of \cite[Theorem 4.3]{Hi1} (also that of Theorem
\ref{T-1.1}\,(1) above). We omit the details.

Finally, to treat the case where $-1\le p,q\le0$ and $1/\min\{p,q\}\le s<0$, we can use
the same technique as in the last paragraph of the proof of Theorem \ref{T-1.1}\,(1).
Consider $\hat\Phi,\hat\Psi$ as given there and the adjoint operator mean
$A\,\sigma^*\,B:=(A^{-1}\,\sigma\,B^{-1})^{-1}$ \cite{KA}; then we notice that
$$
\bigl\{\hat\Phi(A^p)\,\sigma^*\,\hat\Psi(B^q)\bigr\}^s
=\{\Phi(A^{-p})\,\sigma\,\Psi(B^{-q})\}^{-s}.
$$
Hence it remains to see that the joint concavity assertion for $\Phi,\Psi$ is valid for
$\hat\Phi,\hat\Psi$ too. Indeed, inequality \eqref{F-2.1} follows from Lemma \ref{L-2.4}
below and all other arguments in the above proof can be repeated with only use of
monotonicity of $\hat\Phi,\hat\Psi$. Moreover, the proof of \cite[Theorem 4.3]{Hi1} can
easily be modified to obtain joint concavity of
$\Tr\bigl\{\hat\Phi(A^p)\,\sigma\,\hat\Psi(B^p)\bigr\}^{1/p}$ when $0<p\le1$.
\end{proof}

\begin{lemma}\label{L-2.4}
Let $\Phi:\bM_n\to\bM_l$ be a strictly positive linear map. If $0\le p\le1$, then the
function $A\in\bP_n\mapsto\Phi(A^{-p})^{-1}=\hat\Phi(A^p)$ is operator concave.
\end{lemma}

\begin{proof}
Recall \cite[Corollary 3.2]{An} that $A\in\bP_n\mapsto\Phi(A^{-1})^{-1}$ is operator
concave. For every $A,B\in\bP_n$, since $((A+B)/2)^{-p}\le((A^p+B^p)/2)^{-1}$, we have
$$
\Phi\biggl(\biggl({A+B\over2}\biggr)^{-p}\biggr)^{-1}
\ge\Phi\biggl(\biggl({A^p+B^p\over2}\biggr)^{-1}\biggr)^{-1}
\ge{\Phi(A^{-p})^{-1}+\Phi(B^{-p})^{-1}\over2}.
$$
\end{proof}

\noindent
{\it Proof of Theorem \ref{T-2.2}.}\enspace
(1)\enspace
For every $A_1,A_2\in\bM_n^+$ and $B_1,B_2\in\bM_m^+$ and every Ky Fan $k$-anti-norm
$\|\cdot\|_{\{k\}}$, $1\le k\le l$, there exists a rank $k$ projection $E$ commuting with
$\Phi(((A_1+A_2)/2)^p)\,\sigma\,\Psi(((B_1+B_2)/2)^q)$ such that
\begin{align}
&\bigg\|\biggl\{\Phi\biggl(\biggl({A_1+A_2\over2}\biggr)^p\biggr)
\,\sigma\,\Psi\biggl(\biggl({B_1+B_2\over2}\biggr)^q\biggr)\biggr\}^s\bigg\|_{\{k\}}
\nonumber\\
&\quad=\Tr\biggl\{E\biggl(\Phi\biggl(\biggl({A_1+A_2\over2}\biggr)^p\biggr)
\,\sigma\,\Psi\biggl(\bigg({B_1+B_2\over2}\biggr)^q\biggr)\biggr)E\biggr\}^s \nonumber\\
&\quad=\lim_{\eps\searrow0}
\Tr\biggl\{(E+\eps I_l)\biggl(\Phi\biggl(\biggl({A_1+A_2\over2}\biggr)^p\biggr)
\,\sigma\,\Psi\biggl(\bigg({B_1+B_2\over2}\biggr)^q\biggr)\biggr)(E+\eps I_l)\biggr\}^s
\nonumber\\
&\quad=\lim_{\eps\searrow0}
\Tr\biggl\{\biggl((E+\eps I_l)\biggl(\Phi\biggl(\biggl({A_1+A_2\over2}\biggr)^p\biggr)
(E+\eps I_l)\biggr) \nonumber\\
&\qquad\qquad\qquad\qquad
\,\sigma\,\biggl((E+\eps I_l)\Psi\biggl(\bigg({B_1+B_2\over2}\biggr)^q\biggr)\biggr)
(E+\eps I_l)\biggr)\biggr\}^s \label{F-2.2}
\end{align}
due to the transformer equality for $\sigma$ \cite{KA}. Apply Lemma \ref{L-2.3} to the
positive linear maps $(E+\eps I_l)\Phi(\cdot)(E+\eps I_l)$ and
$(E+\eps I_l)\Psi(\cdot)(E+\eps I_l)$ to obtain
\begin{align}
&\Tr\biggl\{\biggl((E+\eps I_l)\biggl(\Phi\biggl(\biggl({A_1+A_2\over2}\biggr)^p\biggr)
(E+\eps I_l)\biggr) \nonumber\\
&\qquad\qquad\qquad
\,\sigma\,\biggl((E+\eps I_l)\Psi\biggl(\bigg({B_1+B_2\over2}\biggr)^q\biggr)\biggr)
(E+\eps I_l)\biggr)\biggr\}^s \nonumber\\
&\quad\ge{1\over2}\Bigl(\Tr\bigl\{\bigl((E+\eps I_l)\Phi(A_1^p)(E+\eps I_l)\bigr)
\,\sigma\,\bigl((E+\eps I_l)\Psi(B_1^q)(E+\eps I_l)\bigr)\bigr\}^s \nonumber\\
&\qquad\qquad\qquad
+\Tr\bigl\{\bigl((E+\eps I_l)\Phi(A_2^p)(E+\eps I_l)\bigr)
\,\sigma\,\bigl((E+\eps I_l)\Psi(B_2^q)(E+\eps I_l)\bigr)\bigr\}^s\Bigr) \nonumber\\
&\quad={1\over2}\Bigl(\Tr\bigl\{(E+\eps I_l)\bigl(\Phi(A_1^p)\,\sigma\,\Psi(B_1^q)\bigr)
(E+\eps I_l)\bigr\}^s \nonumber\\
&\qquad\qquad\qquad
+\Tr\bigl\{(E+\eps I_l)\bigl(\Phi(A_2^p)\,\sigma\,\Psi(B_2^q)\bigr)(E+\eps I_l)
\bigr\}^s\Bigr) \nonumber\\
&\quad\longrightarrow{1\over2}\Bigl(
\Tr\bigl\{E\bigl(\Phi(A_1^p)\,\sigma\,\Psi(B_1^q)\bigr)E\bigr\}^s
+\Tr\bigl\{E\bigl(\Phi(A_2^p)\,\sigma\,\Psi(B_2^q)\bigr)E\bigr\}^s\Bigr)
\quad\mbox{as $\eps\searrow0$}. \label{F-2.3}
\end{align}

For each $C\in\bM_l^+$, besides $\lambda_j^\uparrow(C)$ and $\lambda_j^\downarrow(C)$,
$1\le j\le l$, we write $\lambda_j^\uparrow(ECE)$ and $\lambda_j^\downarrow(ECE)$,
$1\le j\le k$, for the eigenvalues (in increasing and decreasing order, respectively) of
$ECE|_{E\bC^l}$ regarded as an element of $\bM_k^+$. Note that
\begin{equation}\label{F-2.4}
\lambda_j^\uparrow(ECE)\ge\lambda_j^\uparrow(C),\quad
\lambda_j^\downarrow(ECE)\le\lambda_j^\downarrow(C),\qquad j=1,\dots,k.
\end{equation}
Hence we have for $s>0$,
\begin{align*}
\Tr\bigl\{E\bigl(\Phi(A_1^p)\,\sigma\,\Psi(B_1^q)\bigr)E\bigr\}^s
&=\sum_{j=1}^k\bigl\{
\lambda_j^\uparrow\bigl(E\bigl(\Phi(A_1^p)\,\sigma\,\Psi(B_1^q)\bigr)E\bigr)\bigr\}^s \\
&\ge\sum_{j=1}^k\bigl\{
\lambda_j^\uparrow\bigl(\Phi(A_1^p)\,\sigma\,\Psi(B_1^q)\bigr)\bigr\}^s
=\big\|\bigl\{\Phi(A_1^p)\,\sigma\,\Psi(B_1^q)\bigr\}^s\big\|_{\{k\}},
\end{align*}
and the same inequality follows for $s<0$ as well (by replacing $\lambda_j^\uparrow$ with
$\lambda_j^\downarrow$). Similarly
$$
\Tr\bigl\{E\bigl(\Phi(A_2^p)\,\sigma\,\Psi(B_2^q)\bigr)E\bigr\}^s
\ge\big\|\bigl\{\Phi(A_2^p)\,\sigma\,\Psi(B_2^q)\bigr\}^s\big\|_{\{k\}}.
$$
Combining these with \eqref{F-2.2} and \eqref{F-2.3} yields that
\begin{align*}
&\bigg\|\biggl\{\Phi\biggl(\biggl({A_1+A_2\over2}\biggr)^p\biggr)
\,\sigma\,\Psi\biggl(\biggl({B_1+B_2\over2}\biggr)^q\biggr)\biggr\}^s\bigg\|_{\{k\}} \\
&\qquad\ge{1\over2}
\Bigl(\big\|\bigl\{\Phi(A_1^p)\,\sigma\,\Psi(B_1^q)\bigr\}^s\big\|_{\{k\}}
+\big\|\bigl\{\Phi(A_2^p)\,\sigma\,\Psi(B_2^q)\bigr\}^s\big\|_{\{k\}}\Bigr) \\
&\qquad={1\over2}
\Big\|\bigl(\bigl\{\Phi(A_1^p)\,\sigma\,\Psi(B_1^q)\bigr\}^s\bigr)^\uparrow
+\bigl(\bigl\{\Phi(A_1^p)\,\sigma\,\Psi(B_1^q)\bigr\}^s\bigr)^\uparrow\Big\|_{\{k\}},
\end{align*}
where $C^\uparrow$ for $C\in\bM_l^+$ means the diagonal matrix
$\mathrm{diag}(\lambda_1^\uparrow(C),\dots,\lambda_l^\uparrow(C))$. Therefore, for any
anti-norm,
\begin{align*}
&\bigg\|\biggl\{\Phi\biggl(\biggl({A_1+A_2\over2}\biggr)^p\biggr)
\,\sigma\,\Psi\biggl(\biggl({B_1+B_2\over2}\biggr)^q\biggr)\biggr\}^s\bigg\|_! \\
&\qquad\ge
{1\over2}\Big\|\bigl(\bigl\{\Phi(A_1^p)\,\sigma\,\Psi(B_1^q)\bigr\}^s\bigr)^\uparrow
+\bigl(\bigl\{\Phi(A_1^p)\,\sigma\,\Psi(B_1^q)\bigr\}^s\bigr)^\uparrow\Big\|_! \\
&\qquad\ge{1\over2}\Bigl(\big\|\bigl\{\Phi(A_1^p)\,\sigma\,\Psi(B_1^q)\bigr\}^s\big\|_!
+\big\|\bigl\{\Phi(A_1^p)\,\sigma\,\Psi(B_1^q)\bigr\}^s\big\|_!\Bigr)
\end{align*}
by the Ky Fan dominance principle \cite[Lemma 4.2]{BH1}, the superadditivity property,
and unitary conjugation invariance of anti-norms.

(2) is immediately seen by applying (1) to the derived anti-norm $\|A^{-1}\|^{-1}$ for
$A\in\bP_l$ \cite[Proposition 4.6]{BH2}.\qed

\bigskip
From the particular case of Theorem \ref{T-2.2} where $\Phi=\Psi=\mathrm{id}$, we have

\begin{cor}\label{C-2.5}
Let $\sigma$ be any operator mean. Assume that either $0\le p,q\le1$ and
$0<s\le1/\max\{p,q\}$, or $-1\le p,q\le0$ and $1/\min\{p,q\}\le s<0$.
\begin{itemize}
\item[\rm(1)] For any symmetric anti-norm $\|\cdot\|_!$ on $\bM_n^+$,
$\|(A^p\,\sigma\,B^q)^s\|_!$ is jointly concave in $A,B\in\bP_n$.
\item[\rm(2)] For any symmetric norm $\|\cdot\|$ on $\bM_n$,
$\|(A^p\,\sigma\,B^q)^{-s}\|^{-1}$ is jointly concave in $A,B\in\bP_n$, and hence
$\|(A^p\,\sigma\,B^q)^{-s}\|$ is jointly convex in $A,B\in\bP_n$.
\end{itemize}
\end{cor}

In \cite{CL1} Carlen and Lieb proved that the Minkowski type trace function
\begin{equation}\label{F-2.5}
(A,B)\in\bM_n^+\times\bM_n^+\longmapsto\Tr(A^p+B^p)^{1/p}
\end{equation}
is jointly concave if $0<p\le1$, jointly convex if $p=2$, and not jointly convex (also
not jointly concave) if $p>2$. The latter assertions when $p=2$ and when $p>2$ were also
shown in \cite{AnHi}, and the former when $0<p\le1$ was a bit generalized in
\cite[Theorem 2.1]{Hi1}. Bekjan \cite{Be} later treated joint concavity/convexity of trace
functions complementing \eqref{F-2.5} and proved that when $0<p\le1$,
$\Tr(A^{-p}+B^{-p})^{-1/p}$ is jointly concave in $(A,B)\in\bP_n\times\bP_n$, and
$\Tr(A^{-p}+B^{-p})^{1/p}$ and $\Tr(A^p+B^p)^{-1/p}$ are jointly convex in
$(A,B)\in\bP_n\times\bP_n$. Furthermore, in the second paper \cite{CL2} of the same title,
Carlen and Lieb affirmatively settled the conjecture that \eqref{F-2.5} is jointly convex
if $1\le p\le2$. (For the trace function \eqref{F-2.5} see also the remark
around \eqref{F-3.2} in the next section.)

The above mentioned trace inequalities in \cite{Be} are generalized by the following
special case of Theorem \ref{T-2.2} where $\sigma$ is the arithmetic mean and both
$\|\cdot\|$ and $\|\cdot\|_!$ are the trace functional.

\begin{cor}\label{C-2.6}
Let $\Phi$ and $\Psi$ be as in Theorem \ref{T-2.2}. Under the same assumption of $p$, $q$
and $s$ as in Theorem \ref{T-2.2}, $\Tr\{\Phi(A^p)+\Psi(B^q)\}^s$ and
$\bigl(\Tr\{\Phi(A^p)+\Psi(B^q)\}^{-s}\bigr)^{-1}$ are jointly concave in
$(A,B)\in\bP_n\times\bP_m$, and hence $\Tr\{\Phi(A^p)+\Psi(B^q)\}^{-s}$ is jointly convex
in $(A,B)\in\bP_n\times\bP_m$.
\end{cor}

\begin{cor}
Assume that $\Phi:\bM_n\to\bM_l$ and $\Psi:\bM_m\to\bM_l$ are positive linear maps such
that $\Phi(I_n)+\Psi(I_m)=I_l$.
\begin{itemize}
\item[\rm(1)] For any symmetric anti-norm $\|\cdot\|_!$ on $\bM_l^+$, the function
$$
(A,B)\in\bP_n\times\bP_m\longmapsto\|\exp\{\Phi(\log A)+\Psi(\log B)\}\|_!
$$
is jointly concave.
\item[\rm(2)] For any symmetric norm $\|\cdot\|$ on $\bM_l$, the function
$$
(A,B)\in\bP_n\times\bP_m\longmapsto\|\exp\{\Phi(-\log A)+\Psi(-\log B)\}\|^{-1}
$$
is jointly concave, and hence the function
$$
(A,B)\in\bP_n\times\bP_m\longmapsto\|\exp\{\Phi(-\log A)+\Psi(-\log B)\}\|
$$
is jointly convex.
\end{itemize}
\end{cor}

Indeed, we may assume by continuity that $\Phi$ and $\Psi$ are strictly positive linear
maps. Then the assertion (1) follows from Theorem \ref{T-2.2}\,(1) with $\sigma$ the
arithmetic mean since
$$
\lim_{p\searrow0}\{\Phi(A^p)+\Psi(B^p)\}^{1/p}=\exp\{\Phi(\log A)+\Psi(\log B)\}.
$$
The assertion (2) is a consequence of (1) and \cite[Proposition 4.6]{BH2} as before. By
choosing the Minkowski functional $\det^{1/n}$ as $\|\cdot\|_!$, the above (1) implies
that
$$
(A,B)\in\bP_n\times\bP_m\mapsto\exp\{\tau(\Phi(\log A)+\Psi(\log B))\}
$$
is jointly concave, where $\tau$ is the normalized trace on $\bM_l$. This is similar to
\cite[Theorem 5.2]{BH1}.

\section{Norm functions of Epstein type}

In this section we deal with convexity or concavity properties for symmetric norm or
anti-norm functions of the form $\|\Phi(A^p)^s\|$ or $\|\Phi(A^p)^s\|_!$ where $\Phi$ is a
(strictly) positive linear map. In particular, when $\Phi(A):=X^*AX$ and $\|\cdot\|$
(or $\|\cdot\|_!$) is the trace functional, the function is
\begin{equation}\label{F-3.1}
A\in\bP_n\longmapsto\Tr(X^*A^pX)^s,
\end{equation}
whose concavity when $0<p\le1$ and $s=1/p$ was established by Epstein \cite{Ep} (see
\cite{Hi1} for some generalizations). In \cite{CL2} it was proved that the function
\eqref{F-3.1} is, for any $X\in\bM_n$, convex if $1\le p\le2$ and $s\ge1/p$, concave if
$0<p\le1$ and $1\le s\le1/p$, and neither convex nor concave if $p>2$. It was also
pointed out there that joint concavity of \eqref{F-2.5} when $0<p\le1$ is easily seen
from Epstein's concavity \cite{Ep} since
\begin{equation}\label{F-3.2}
\Tr\biggl(\bmatrix I&0\\I&0\endbmatrix^*\bmatrix A&0\\0&B\endbmatrix^p
\bmatrix I&0\\I&0\endbmatrix\biggr)^{1/p}=\Tr(A^p+B^p)^{1/p}.
\end{equation}

In the rest of the section we assume that $p$ and $s$ are non-zero; otherwise, the
assertion is trivial.

\begin{thm}\label{T-3.1}
Let $n,m\in\bN$, and $\Phi:\bM_n\to\bM_m$ be a strictly positive linear map.
\begin{itemize}
\item[\rm(1)] Let $\|\cdot\|_!$ be any symmetric anti-norm on $\bM_m^+$. If either
$0<p\le1$ and $0<s\le1/p$, or $-1\le p<0$ and $1/p\le s<0$, then
$A\in\bP_n\mapsto\|\Phi(A^p)^s\|_!$ is concave.
\item[\rm(2)] Let $\|\cdot\|$ be any symmetric norm on $\bM_m$. If either $0<p\le1$ and
$0<s\le1/p$, or $-1\le p<0$ and $1/p\le s<0$, then
$A\in\bP_n\mapsto\|\Phi(A^p)^{-s}\|^{-1}$ is concave. Furthermore,
$A\in\bP_n\mapsto\|\Phi(A^p)^s\|$ is convex if one of the following three conditions is
satisfied:
\begin{equation}\label{F-3.3}
\begin{cases}
-1\le p<0\ \mbox{and}\ s>0, \\
0<p\le1\ \mbox{and}\ s<0, \\
1\le p\le2\ \mbox{and}\ s\ge1.
\end{cases}
\end{equation}
\end{itemize}
\end{thm}

\begin{proof}
(1) and the first assertion of (2) are included in Theorem \ref{T-2.2} as special cases
where $B=A$, $\Psi=\Phi$ and $q=p$. For the second assertion of (2) it remains to show
that $A\in\bP_n\mapsto\|\Phi(A^p)^s\|$ is convex when $p\in[-1,0)\cup[1,2]$ and $s\ge1$
and when $p\in(0,1]$ and $s<0$. This can easily be verified as follows: For the latter
case, since $x^p$ is operator concave, we have
$\Phi(((A+B)/2)^p)\ge(\Phi(A^p)+\Phi(B^p))/2$ for $A,B\in\bP_n$. Hence we need to show
that $A\in\bP_n\mapsto\|A^s\|$ is decreasing and convex when $s<0$. If $A\le B$, then
$(A^s)^\downarrow\ge(B^s)^\downarrow$ so that $\|A^s\|\ge\|B^s\|$. For $A,B\in\bP_n$,
from the Ky Fan majorization $((A+B)/2)^\downarrow\prec(A^\downarrow+B^\downarrow)/2$
we have
\begin{align*}
\biggl(\biggl({A+B\over2}\biggr)^s\biggr)^\uparrow
&=\biggl(\biggl({A+B\over2}\biggr)^\downarrow\biggr)^s
\prec_w\biggl({A^\downarrow+B^\downarrow\over2}\biggr)^s \\
&\le{(A^\downarrow)^s+(B^\downarrow)^s\over2}
={(A^s)^\uparrow+(B^s)^\uparrow\over2}
\end{align*}
so that
$$
\bigg\|\biggl({A+B\over2}\biggr)^s\bigg\|
\le\bigg\|{(A^s)^\uparrow+(B^s)^\uparrow\over2}\bigg\|
\le{\|A^s\|+\|B^s\|\over2}.
$$
The former case is similarly shown.
\end{proof}

The above theorem does not cover the convexity assertion in \cite{CL2} for \eqref{F-3.1}
when $1\le p\le2$ and $1/p\le s<1$. But we can extend this by using the
variational method in \cite{CL2} itself in the following way. Here, we assume
a stronger assumption of $\Phi$ being completely positive (CP).

\begin{thm}\label{T-3.2}
Let $n,m\in\bN$, and $\Phi:\bM_n\to\bM_m$ be a CP linear map. Let $\|\cdot\|$ be any
symmetric norm on $\bM_m$. If $1\le p\le2$ and $s\ge1/p$, then
$A\in\bM_n^+\mapsto\|\Phi(A^p)^s\|$ is convex.
\end{thm}

\begin{proof}
By Theorem \ref{T-3.1}\,(2) we may assume that $1/p\le s\le1$. First, we prove the trace
function case. This part of the proof is a slight modification of that in \cite{CL2}. Let
$r:=1/s$ so that $1\le r\le p\le2$. By \cite[Lemma 2.2]{CL2} we have
$$
\Tr\Phi(A^p)^{1/r}
={1\over r}\inf_{B\in\bP_m}\Tr\{\Phi(A^p)B^{1-r}+(r-1)B\}.
$$
Hence it suffices to show that
$$
(A,B)\in\bP_n\times\bP_m\longmapsto\Tr\Phi(A^p)B^{1-r}
$$
is jointly convex. Since $-1\le1-r\le0$ and $1-(1-r)\le p\le2$, it follows from
\cite[Corollary 6.3]{An} that
$(A,B)\in\bP_n\times\bP_m\mapsto A^p\otimes B^{1-r}\in\bM_m\otimes\bM_m$ is jointly convex.
Hence so is
$$
(A,B)\in\bP_n\times\bP_m\longmapsto(\Phi\otimes\id)(A^p\otimes B^{1-r})
=\Phi(A^p)\otimes B^{1-r}\in\bM_m\times\bM_m,
$$
since $\Phi\otimes\id$ is positive by the CP assumption of $\Phi$. This implies the
assertion for the trace function.

Next, we extend the result to the symmetric norm function. Since the proof is similar to
(and easier than) that of Theorem \ref{T-2.2}\,(1), we only sketch it. For every
$A,B\in\bM_n^+$ and every Ky Fan $k$-norm $\|\cdot\|_{(k)}$, $1\le k\le m$, there exists a
rank $k$ projection $E$ commuting with $\Phi(((A+B)/2)^p)$ such that
$$
\bigg\|\Phi\biggl(\biggl({A+B\over2}\biggr)^p\biggr)^s\bigg\|_{(k)}
=\Tr\biggl\{E\Phi\biggl(\biggl({A+B\over2}\biggr)^p\biggr)E\biggr\}^s.
$$
Applying the trace function case to the CP linear map $E\Phi(\cdot)E$ we have
$$
\Tr\biggl\{E\Phi\biggl(\biggl({A+B\over2}\biggr)^p\biggr)E\biggr\}^s
\le{1\over2}\bigl(\Tr\{E\Phi(A^p)E\}^s+\Tr\{E\Phi(B^p)E\}^s\bigr).
$$
Moreover, using the second inequality of \eqref{F-2.4} we have
$\Tr\{E\Phi(A^p)E\}^s\le\|\Phi(A^p)^s\|_{(k)}$ and
$\Tr\{E(\Phi(B^p)E\}^s\le\|\Phi(B^p)^s\|_{(k)}$ so that
$$
\bigg\|\Phi\biggl(\biggl({A+B\over2}\biggr)^p\biggr)^s\bigg\|_{(k)}
\le{1\over2}\bigl(\|\Phi(A^p)^s\|_{(k)}+\|\Phi(B^p)^s\|_{(k)}\bigr),
$$
which implies the desired convexity assertion.
\end{proof}

Complementing Corollary \ref{C-2.6} we give

\begin{cor}
Let $n,m,l\in\bN$. Let $\Phi:\bM_n\to\bM_l$ and $\Psi:\bM_m\to\bM_l$ be CP linear maps.
Assume that $1\le p\le2$ and $s\ge1/p$. Then for any symmetric norm
$\|\cdot\|$ on $\bM_l$, the function
$$
(A,B)\in\bM_n^+\times\bM_m^+\longmapsto\|\{\Phi(A^p)+\Psi(B^p)\}^s\|
$$
is jointly convex, and in particular so is $\Tr\{\Phi(A^p)+\Psi(B^p)\}^s$ in
$(A,B)\in\bM_n^+\times\bM_m^+$.
\end{cor}

\begin{proof}
The proof is a slight modification of expression \eqref{F-3.2}. Consider a linear map
$\Theta:\bM_{n+m}\to\bM_l$ defined by
$$
\Theta\biggl(\begin{bmatrix}A&X\\Y&B\end{bmatrix}\biggr):=\Phi(A)+\Psi(B)
$$
in the form of block matrices with $A\in\bM_n$ and $B\in\bM_m$. It is easy to see that
$\Theta$ is CP. Since
$$
\{\Phi(A^p)+\Psi(B^p)\}^s=\Theta\biggl(\begin{bmatrix}A&0\\0&B\end{bmatrix}^p\biggr)^s,
\qquad A\in\bM_n^+,\ B\in\bM_m^+,
$$
the assertion is an immediate consequence of Theorem \ref{T-3.2}.
\end{proof}

\section{Necessary conditions}

In the previous sections we obtained sufficient conditions on the parameters $p,q,s$ (or
$p,s$) for which the relevant matrix trace or norm function is (jointly) concave or convex.
The aim of this section is to specify necessary conditions on the parameters for those
concavity/convexity properties to hold.

Concerning the necessity direction for (joint) concavity of \eqref{F-1.1} and \eqref{F-3.1}
we have

\begin{prop}\label{P-4.1}
\begin{itemize}
\item[\rm(1)] Assume that $p,s\ne0$. If $A\in\bP_2\mapsto\Tr(X^*A^pX)^s$ is concave for
any invertible $X\in\bM_2$, then either $0<p\le1$ and $0<s\le1/p$, or $-1\le p<0$ and
$1/p\le s<0$.
\item[\rm(2)] Assume that $p$, $q$ and $s$ are all non-zero. If
$(A,B)\in\bP_2\times\bP_2\mapsto\Tr(A^{p/2}B^qA^{p/2})^s$ is jointly concave, then either
$0<p,q\le1$ and $0<s\le1/(p+q)$, or $-1\le p,q<0$ and
$1/(p+q)\le s<0$.
\end{itemize}
\end{prop}

\begin{proof}
(1)\enspace First assume that $s>0$. By assumption, $x^{ps}$ is concave in $x>0$ so that
$0<ps\le1$. For every $a,b,\eps>0$ let $A:=\begin{bmatrix}a&0\\0&b\end{bmatrix}$ and
$X_\eps:=\begin{bmatrix}1&0\\1&\eps\end{bmatrix}$; then
$\Tr(X_\eps^*A^pX_\eps)^s\to(a^p+b^p)^s$
as $\eps\searrow0$. So $(a^p+b^p)^s$ is concave in $a,b>0$. Since
\begin{equation}\label{F-4.1}
{d^2\over dx^2}\,(x^p+b)^s=psx^{p-2}(x^p+b)^{s-2}\{(ps-1)x^p+(p-1)b\},
\end{equation}
we must have $(ps-1)x^p+(p-1)b\le0$ for all $x,b>0$, which gives $p\le1$ as well as
$ps\le1$. When $s<0$, the result follows from the above case since
$\Tr(X^*A^pX)^s=\Tr(X^{-1}A^{-p}(X^{-1})^*)^{-s}$.

(2)\enspace As in the proof of (1) it suffices to assume that $s>0$. By assumption,
$x^{(p+q)s}$ is concave in $x>0$ so that $(p+q)s\le1$. The assumption also implies that
$A\in\bP_2\mapsto\Tr(X^*A^pX)^s$ is concave for every invertible $X\in\bM_2$, as readily
seen by taking the polar decomposition of $X$. Hence (1) implies that $0\le p\le1$.
Similarly, $0\le q\le1$.
\end{proof}

\begin{remark}\label{R-4.2}\rm
We have assumed in Proposition \ref{P-4.1}\,(2) that $p,q\ne0$ (stronger than
$(p,q)\ne(0,0)$). When $q=0$, the condition there means concavity of $\Tr A^{ps}$,
implying $0\le ps\le1$. However, for joint concavity of $\Tr(A^{p/2}XB^qX^*A^{p/2})^s$
for any invertible $X\in\bM_2$, we have the same necessary condition as in
Proposition \ref{P-4.1}\,(2) including the case $p=0$ or $q=0$.
\end{remark}

There is no gap between a necessary condition in Proposition \ref{P-4.1}\,(1) and a
sufficient condition in Theorem \ref{T-3.1}\,(1). This says that Theorem \ref{T-3.1}\,(1)
is a best possible result. The difference between a necessary condition in Proposition
\ref{P-4.1}\,(2) and a sufficient condition in Theorem \ref{T-1.1}\,(1) is rather small:
$0<s<1/2$ for $0<p,q\le1$, or $-1/2<s<0$ for $-1\le p,q<0$.
When restricted to $s=1$, a necessary condition for joint concavity of
$(A,B)\in\bP_2\times\bP_2\mapsto\Tr A^pB^q$ is that $0\le p,q\le1$ and $p+q\le1$, which is
also sufficient for joint concavity of \eqref{F-1.2} for any $X\in\bM_n$ (as shown in
\cite{Li,An}) and even for \eqref{F-1.1} with $s=1$.

\begin{remark}\label{R-4.3}\rm
For the case in the gap between conditions of Proposition \ref{P-4.1}\,(2) and of Theorem
\ref{T-1.1}\,(1), the following is worth noting: Assume that $0<p,q\le1$ and $0<s\le1$.
For every positive linear maps $\Phi:\bM_n\to\bM_l$, $\Psi:\bM_m\to\bM_l$ and for every
$A_1,A_2\in\bM_n^+$, $B_1,B_2\in\bM_m^+$ one has
\begin{align*}
&\biggl\{\Phi\biggl(\biggl({A_1+A_2\over2}\biggr)^p\biggr)^{1/2}
\Psi\biggl(\biggl({B_1+B_2\over2}\biggr)^q\biggr)
\Phi\biggl(\biggl({A_1+A_2\over2}\biggr)^p\biggr)^{1/2}\biggr\}^s \\
&\qquad\ge U\biggl\{\biggl({\Phi(A_1^q)+\Phi(A_2^q)\over2}\biggr)^{1/2}
\biggl({\Psi(B_1^p)+\Psi(B_2^p)\over2}\biggr)
\biggl({\Phi(A_1^q)+\Phi(A_2^q)\over2}\biggr)^{1/2}\biggr\}^sU^*
\end{align*}
for some unitary $U\in\bM_l$. Hence, to settle the case $0<s<1/2$ (and $0\le p,q\le1$) of
Theorem \ref{T-1.1}\,(1), we need to prove that
$(A,B)\in\bP_n^+\times\bP_n^+\mapsto\Tr(A^{1/2}BA^{1/2})^s$ is jointly concave if
$0<s<1/2$.
\end{remark}

Concerning the necessity direction for (joint) convexity of \eqref{F-1.1} and \eqref{F-3.1}
we have

\begin{prop}\label{P-4.4}
\begin{itemize}
\item[\rm(1)] Assume that $p,s\ne0$. If $A\in\bP_4\mapsto\Tr(X^*A^pX)^s$ is convex for
every $X\in\bP_4$, then one of the following four conditions is satisfied:
$$
\begin{cases}
-1\le p<0\ \mbox{and}\ s>0, \\
1\le p\le2\ \mbox{and}\ s\ge1/p,
\end{cases}
$$
and their counterparts where $(p,s)$ is replaced with $(-p,-s)$.
\item[\rm(2)] Assume that $p$, $q$ and $s$ are all non-zero. If
$(A,B)\in\bP_4\times\bP_4\mapsto\Tr(A^{p/2}B^qA^{p/2})^s$ is jointly convex, then one of
the following six conditions is satisfied:
$$
\begin{cases}
-1\le p,q<0\ \mbox{and}\ s>0, \\
-1\le p<0,\ 1\le q\le2,\ p+q>0\ \mbox{and}\ s\ge1/(p+q), \\
1\le p\le2,\ -1\le q<0,\ p+q>0\ \mbox{and}\ s\ge1/(p+q),
\end{cases}
$$
and their counterparts where $(p,q,s)$ is replaced with $(-p,-q,-s)$.
\end{itemize}
\end{prop}

\begin{proof}
As in the proof of Proposition \ref{P-4.1} it suffices to assume that $s>0$.

(1)\enspace Let $X_\eps:=\begin{bmatrix}I_2&0\\I_2&\eps I_2\end{bmatrix}\in\bM_4$ for
$\eps>0$. For any $A,B\in\bP_2$, since
$$
\Tr\biggl(X_\eps^*\begin{bmatrix}A&0\\0&B\end{bmatrix}^pX_\eps\biggr)^s
\longrightarrow\Tr(A^p+B^p)^s\quad\mbox{as $\eps\searrow0$},
$$
the assumption implies that $(A,B)\in\bP_2\times\bP_2\mapsto\Tr(A^p+B^p)^s$ is jointly
convex, so $\ffi_t(A):=\Tr(tA^p+B)^s$ is convex in $A\in\bP_2$ for any $t>0$ and
$B\in\bP_2$. Now, the argument below is the same as that in \cite{Be} (also \cite{CL2})
while it is given for completeness. Since
$$
{d\over dt}\,\ffi_t(A)\Big|_{t=0}=s\Tr B^{s-1}A^p,
$$
we notice that
$$
\ffi_t(A)=\Tr B^s+st\Tr B^{s-1}A^p+o(t)\quad\mbox{as $t\searrow0$}.
$$
Therefore, for $A_1,A_2\in\bP_2$ we have
\begin{align*}
0&\ge\ffi_t\biggl({A_1+A_2\over2}\biggr)-{\ffi_r(A_1)+\ffi_r(A_2)\over2} \\
&=st\biggl\{\Tr B^{s-1}\biggl({A_1+A_2\over2}\biggr)^p
-\Tr B^{s-1}\biggl({A_1^p+A_2^p\over2}\biggr)\biggr\}+o(t)
\quad\mbox{as $t\searrow0$}
\end{align*}
so that
$$
\Tr B^{s-1}\biggl({A_1+A_2\over2}\biggr)^p\le\Tr B^{s-1}\biggl({A_1^p+A_2^p\over2}\biggr).
$$
When $s\ne1$, this means that $x^p$ ($x>0$) is matrix convex of order $2$, which is also
clear for $s=1$ from the assumption itself. Hence by \cite[Proposition 3.1]{HT} we must
have $-1\le p\le0$ or $1\le p\le2$. When $1\le p\le2$, $ps\ge1$ since $x^{ps}$ is convex.

(2)\enspace Since the assumption here implies that of (1) thanks to $p,q\ne0$,
it follows that either $-1\le p<0$, or $1\le p\le2$ and $s\ge1/p$. Similarly,
either $-1\le q<0$, or $1\le q\le2$ and $s\ge1/q$. Since $x^{ps}y^{qs}$ is
jointly convex in $x,y>0$, computing the Hessian gives
\begin{equation}\label{F-4.2}
pq\{1-(p+q)s\}\ge0.
\end{equation}
Hence $ps\ge1$ and $qs\ge1$ cannot occur simultaneously, so the following three cases are
possible (when $s>0$):
$$
\begin{cases}
-1\le p,q<0, \\
-1\le p<0$\ \mbox{and}\ $1\le q\le2, \\
1\le p\le2$\ \mbox{and}\ $-1\le q<0.
\end{cases}
$$
For the above second case, \eqref{F-4.2} gives $p+q>0$ and $s\ge1/(p+q)$. The third case
is similar.
\end{proof}

A gap between a necessary condition in Proposition \ref{P-4.4}\,(1) and a sufficient
condition in \eqref{F-3.3} and Theorem \ref{T-3.2} together is not so big: only the case
$-2\le p\le-1$ and $s\le1/p$. But there is quite a big gap between conditions in
Proposition \ref{P-4.4}\,(2) and in Theorem \ref{T-1.1}\,(2). Concerning the
assumption $p,q\ne0$ in Proposition \ref{P-4.4}\,(2) a remark similar to Remark \ref{R-4.2}
is available. When restricted to $s=1$ (and $p,q\ne0$),
Proposition \ref{P-4.4}\,(2) says that a necessary condition for joint convexity of
$(A,B)\in\bP_4\times\bP_4\mapsto\Tr A^pB^q$ is that
$$
\begin{cases}
-1\le p,q<0, \\
-1\le p<0\ \mbox{and}\ 1-p\le q\le2, \\
-1\le q<0\ \mbox{and}\ 1-q\le p\le2,
\end{cases}
$$
which is exactly a necessary and sufficient condition for joint convexity of \eqref{F-1.2}
for any $X\in\bM_n$ (see \cite[p.\ 221, Remark (4)]{An}). In this connection see also
\cite[Theorem 2]{Be} and \cite[Lemma 5.2]{CL2}.

\section{More discussions}

Theorem \ref{T-1.1} was presented for trace functions while we more generally treated
symmetric (anti-) norm functions in Theorems \ref{T-2.2} and \ref{T-3.1}. So it is
desirable to extend Theorem \ref{T-1.1} to joint concavity/convexity of (anti-) norm
functions. The problem can be reduced to joint concavity of the Ky Fan $k$-anti-norm
functions
$$
(A,B)\in\bP_n\times\bP_m\longmapsto
\big\|\bigl\{\Phi(A^p)^{1/2}\Psi(B^q)\Phi(A^p)^{1/2}\bigr\}^s\big\|_{\{k\}}
$$
and joint convexity of the Ky Fan $k$-norm functions
$$
(A,B)\in\bP_n\times\bP_m\longmapsto
\big\|\bigl\{\Phi(A^p)^{1/2}\Psi(B^q)\Phi(A^p)^{1/2}\bigr\}^s\big\|_{(k)}
$$
for $k=1,\dots,l$ in the situation of Theorem \ref{T-1.1}. In this section we examine the
problem in the special case $k=1$. As in Section 1 we assume that $(p,q)\ne(0,0)$ and
$s\ne0$.

\begin{thm}\label{T-5.1}
Let $\Phi:\bM_n\to\bM_l$ and $\Psi:\bM_m\to\bM_l$ be as in Theorem \ref{T-1.1}.
\begin{itemize}
\item[\rm(1)] The function
$$
(A,B)\in\bP_n\times\bP_m\longmapsto
\lambda_l\bigl(\bigl\{\Phi(A^p)^{1/2}\Psi(B^q)\Phi(A^p)^{1/2}\bigr\}^s\bigr)
$$
is jointly concave, where $\lambda_l(C)=\lambda_l^\downarrow(C)$ is the smallest eigenvalue
of $C\in\bP_l$, if one of the following two conditions is satisfied:
\begin{equation}\label{F-5.1}
\begin{cases}
0\le p,q\le1\ \mbox{and}\ 0<s\le1/(p+q), \\
-1\le p,q\le0\ \mbox{and}\ 1/(p+q)\le s<0.
\end{cases}
\end{equation}
\item[\rm(2)] The function
$$
(A,B)\in\bP_n\times\bP_m\longmapsto
\big\|\bigl\{\Phi(A^p)^{1/2}\Psi(B^q)\Phi(A^p)^{1/2}\bigr\}^s\big\|_\infty,
$$
where $\|\cdot\|_\infty$ is the operator norm, is jointly convex if one of the following
six conditions is satisfied:
\begin{equation}\label{F-5.2}
\begin{cases}
-1\le p,q\le0\ \mbox{and}\ s>0, \\
-1\le p\le0,\ 1\le q\le2,\ p+q>0\ \mbox{and}\ s\ge1/(p+q), \\
1\le p\le2,\ -1\le q\le0,\ p+q>0\ \mbox{and}\ s\ge1/(p+q),
\end{cases}
\end{equation}
and their counterparts where $(p,q,s)$ is replaced with $(-p,-q,-s)$.
\end{itemize}
\end{thm}

\begin{proof}
(1)\enspace
First, assume that $0\le p,q\le1$ and $s=1/(p+q)$. We show that
\begin{align}\label{F-5.3}
&\lambda_l\biggl(\biggl\{\Phi\biggl(\biggl({A_1+A_2\over2}\biggr)^p\biggr)^{1/2}
\Psi\biggl(\biggl({B_1+B_2\over2}\biggr)^q\biggr)
\Phi\biggl(\biggl({A_1+A_2\over2}\biggr)^{p/2}\biggr)\biggr\}^s\biggr) \nonumber\\
&\qquad\ge{
\lambda_l\bigl(\{\Phi(A_1^p)^{1/2}\Psi(B_1^q)\Phi(A_1^p)^{1/2}\}^s\bigr)
+\lambda_l\bigl(\{\Phi(A_2^p)^{1/2}\Psi(B_2^q)\Phi(A_2^p)^{1/2}\}^s\bigr)
\over2}
\end{align}
for every $A_1,A_2\in\bP_n$ and $B_1,B_2\in\bP_m$. Set
$$
\alpha_j:=\lambda_l\bigl(\{\Phi(A_j^p)^{1/2}\Psi(B_j^q)\Phi(A_j^p)^{1/2}\}
^{1/(p+q)}\bigr),\qquad j=1,2.
$$
We then have $\Phi(A_j^p)^{1/2}\Psi(B_j^q)\Phi(A_j^p)^{1/2}\ge\alpha_j^{p+q}I$ so that
\begin{equation}\label{F-5.4}
\Psi((\alpha_j^{-1}B_j)^q)\ge\Phi((\alpha_j^{-1}A_j)^p)^{-1},\qquad j=1,2.
\end{equation}
Since $x^q$ is operator concave, we have
$$
\biggl({B_1+B_2\over\alpha_1+\alpha_2}\biggr)^q
\ge{\alpha_1\over\alpha_1+\alpha_2}(\alpha_1^{-1}B_1)^q
+{\alpha_2\over\alpha_1+\alpha_2}(\alpha_2^{-1}B_2)^q,
$$
which implies that
\begin{align}
\Psi\biggl(\biggl({B_1+B_2\over\alpha_1+\alpha_2}\biggr)^q\biggr)
&\ge{\alpha_1\over\alpha_1+\alpha_2}\Phi((\alpha_1^{-1}A_1)^p)^{-1}
+{\alpha_2\over\alpha_1+\alpha_2}\Phi((\alpha_2^{-1}A_2)^p)^{-1} \nonumber\\
&\ge\biggl\{\Phi\biggl({\alpha_1\over\alpha_1+\alpha_2}(\alpha_1^{-1}A_1)^p
+{\alpha_2\over\alpha_1+\alpha_2}(\alpha_2^{-1}A_2)^p\biggr)\biggr\}^{-1} \label{F-5.5}
\end{align}
due to \eqref{F-5.4} and operator convexity of $x^{-1}$. Moreover, operator concavity
of $x^p$ gives
$$
{\alpha_1\over\alpha_1+\alpha_2}(\alpha_1^{-1}A_1)^p
+{\alpha_2\over\alpha_1+\alpha_2}(\alpha_2^{-1}A_2)^p
\le\biggl({A_1+A_2\over\alpha_1+\alpha_2}\biggr)^p.
$$
Inserting this into \eqref{F-5.5} we have
\begin{equation}\label{F-5.6}
\Psi\biggl(\biggl({B_1+B_2\over\alpha_1+\alpha_2}\biggr)^q\biggr)
\ge\Phi\biggl(\biggl({A_1+A_2\over\alpha_1+\alpha_2}\biggr)^p\biggr)^{-1}
\end{equation}
so that
$$
\Psi\biggl(\biggl({B_1+B_2\over2}\biggr)^q\biggr)
\ge\biggl({\alpha_1+\alpha_2\over2}\biggr)^{p+q}
\Phi\biggl(\biggl({A_1+A_2\over2}\biggr)^p\biggr)^{-1}.
$$
Therefore,
$$
\Phi\biggl(\biggl({A_1+A_2\over2}\biggr)^p\biggr)^{1/2}
\Psi\biggl(\biggl({B_1+B_2\over2}\biggr)^q\biggr)
\Phi\biggl(\biggl({A_1+A_2\over2}\biggr)^p\biggr)^{1/2}
\ge\biggl({\alpha_1+\alpha_2\over2}\biggr)^{p+q}I,
$$
which implies the assertion when $s=1/(p+q)$. The assertion when $0<s<1/(p+q)$ follows by
taking the $s(p+q)$-powers of both sides of \eqref{F-5.3} when $s=1/(p+q)$ and by applying
concavity of $x^{s(p+q)}$.

For the other case where $-1\le p,q\le0$ and $1/(p+q)\le s<0$, we may check (see
\eqref{F-1.10}) that the above proof for the first case can be performed with $\hat\Phi$
and $\hat\Psi$ in place of $\Phi$ and $\Psi$. A non-trivial part is to prove \eqref{F-5.6}
for $\hat\Phi$ and $\hat\Psi$ from \eqref{F-5.4} for $\hat\Phi$ and $\hat\Psi$, which can
be done by Lemma \ref{L-2.4} as follows:
\begin{align*}
\hat\Psi\biggl(\biggl({B_1+B_2\over\alpha_1+\alpha_2}\biggr)^q\biggr)
&\ge{\alpha_1\over\alpha_1+\alpha_2}\hat\Phi((\alpha_1^{-1}A_1)^p)^{-1}
+{\alpha_2\over\alpha_1+\alpha_2}\hat\Phi((\alpha_2^{-1}A_2)^p)^{-1} \\
&=\Phi\biggl({\alpha_1\over\alpha_1+\alpha_2}(\alpha_1^{-1}A_1)^{-p}
+{\alpha_2\over\alpha_1+\alpha_2}(\alpha_2^{-1}A_2)^{-p}\biggr) \\
&\ge\Phi\biggl(\biggl({A_1+A_2\over\alpha_1+\alpha_2}\biggr)^{-p}\biggr)
=\hat\Phi\biggl(\biggl({A_1+A_2\over\alpha_1+\alpha_2}\biggr)^p\biggr)^{-1}.
\end{align*}

(2)\enspace
Assume that $s>0$. When $-1\le p,q\le0$ and $0<s\le-1/(p+q)$, since
$$
\big\|\bigl\{\Phi(A^p)^{1/2}\Psi(B^q)\Phi(A^p)^{1/2}\bigr\}^s\big\|_\infty
=\lambda_l\bigl(\bigl\{\Phi(A^p)^{1/2}\Psi(B^q)\Phi(A^p)^{1/2}\bigr\}^{-s}\bigr)^{-1},
$$
the assertion is immediate from (1) above. When $-1\le p,q\le0$ and $s>-1/(p+q)$, since
$-s(p+q)>1$ and
$$
\big\|\bigl\{\Phi(A^p)^{1/2}\Psi(B^q)\Phi(A^p)^{1/2}\bigr\}^s\big\|_\infty
=\big\|\bigl\{\Phi(A^p)^{1/2}\Psi(B^q)\Phi(A^p)^{1/2}\bigr\}^{-1/(p+q)}
\big\|_\infty^{-s(p+q)},
$$
the assertion follows from the case $s=-1/(p+q)$. Next, assume that $(p,q,s)$ satisfies
the second condition in \eqref{F-5.2}. As in the above argument it suffices to show the
joint convexity assertion when $s=1/(p+q)$. Set
$$
\alpha_j:=\big\|\bigl\{\Phi(A_j^p)^{1/2}\Psi(B_j^q)\Phi(A_j^p)^{1/2}\bigr\}
^{1/(p+q)}\big\|_\infty,\qquad j=1,2.
$$
Then $\Phi(A_j^p)^{1/2}\Psi(B_j^q)\Phi(A_j^p)^{1/2}\le\alpha_j^{p+q}I$, and by the same
argument as in the proof of (1) with use of Lemma \ref{L-2.4} we have
$$
\Phi\bigg(\bigg({A_1+A_2\over2}\biggr)^p\biggr)^{1/2}
\Psi\biggl(\bigg({B_1+B_2\over2}\biggr)^q\biggr)
\Phi\bigg(\bigg({A_1+A_2\over2}\biggr)^p\biggr)^{1/2}
\le\biggl({\alpha_1+\alpha_2\over2}\biggr)^{p+q}I,
$$
which implies the desired joint convexity. Since
$$
\big\|\bigl\{\Phi(A^p)^{1/2}\Psi(B^q)\Phi(A^p)^{1/2}\bigr\}^s\big\|_\infty
=\big\|\bigl\{\Psi(B^q)^{1/2}\Phi(A^p)\Phi(B^q)^{1/2}\bigr\}^s\big\|_\infty,
$$
the assertion holds also when $(p,q,s)$ satisfies the third condition in \eqref{F-5.2}.

Finally, the above proof can be repeated with $\hat\Phi$ and $\hat\Psi$ in place of $\Phi$
and $\Psi$ (while we omit the details), which shows the assertion under the other three
conditions where $s<0$.
\end{proof}

By Theorems \ref{T-1.1} and \ref{T-5.1} we see that Theorem \ref{T-1.1} can be extended to
symmetric (anti-) norm functions in particular when $l=2$ (under the same assumption for
each of (1) and (2) of Theorem \ref{T-1.1}).

Although the next theorem is concerned with functions of the particular form with
$\Phi=\Psi=\id$ in Theorem \ref{T-5.1} under $p,q\ne0$ (stronger than $(p,q)\ne(0,0)$),
it has an advantage that the condition on the parameters is a necessary and sufficient
condition. The theorem indeed says that the conditions on $p$, $q$ and $s$ in
Theorem \ref{T-5.1} are best possible except the case where $p=0$ or $q=0$. When $q=0$
(and $p\ne0$), the concavity/convexity properties in Theorem \ref{T-5.1} are reduced to
concavity of $A\in\bP_n\mapsto\lambda_m(\Phi(A^p)^s)$ and convexity of
$A\in\bP_n\mapsto\|\Phi(A^p)^s\|_\infty$ for every strictly positive linear map
$\Phi:\bM_n\to\bM_m$, which are special cases of the concavity/convexity properties of
Theorem \ref{T-3.1}.

\begin{thm}\label{T-5.2}
Assume that $p,q,s\ne0$.
\begin{itemize}
\item[\rm(1)] The function
\begin{equation}\label{F-5.7}
(A,B)\in\bP_n\times\bP_n\longmapsto\lambda_n\bigl((A^{p/2}B^qA^{p/2})^s\bigr)
\end{equation}
is jointly concave for every $n\in\bN$ (or equivalently, for fixed $n=2$) if
and only if $p$, $q$ and $s$ satisfy one of the conditions in \eqref{F-5.1}.
\item[\rm(2)] The function
\begin{equation}\label{F-5.8}
(A,B)\in\bP_n\times\bP_n\longmapsto\big\|(A^{p/2}B^qA^{p/2})^s\big\|_\infty
\end{equation}
is jointly convex for every $n\in\bN$ (or equivalently, for fixed $n=2$) if
and only if $p$, $q$ and $s$ satisfy one of the conditions in \eqref{F-5.2} and their
counterparts for $(-p,-q,-s)$ in place of $(p,q,s)$.
\end{itemize}
\end{thm}

\begin{proof}
Obviously, the ``if\," parts of (1) and (2) are included in those of Theorem \ref{T-5.1}.
In the rest we prove the ``only if\," parts under $p,q,s\ne0$. As in the proofs of
Propositions \ref{P-4.1} and \ref{P-4.4} we may assume that $s>0$.

(1)\enspace
Since
$$
\bigl\{\lambda_n\bigl((A^{p/2}B^qA^{p/2})^s\bigr)\bigr\}^{-1}
=\big\|(A^{-p/2}B^{-q}A^{-p/2})^s\big\|_\infty,
$$
joint concavity of \eqref{F-5.7} implies joint convexity of \eqref{F-5.8} for
$(-p,-q,s)$. By applying the conclusion of (2) (proved below) to $(-p,-q,s)$ (with $s>0$)
we have
$$
\begin{cases}
0<p,q\le1, \\
0<p\le1,\ -2\le q\le-1,\ p+q<0\ \mbox{and}\ s\ge-1/(p+q), \\
-2\le p\le-1,\ 0<q\le1,\ p+q<0\ \mbox{and}\ s\ge-1/(p+q).
\end{cases}
$$
However, joint concavity of \eqref{F-5.7} implies that $ps,qs,(p+q)s\in(0,1]$ and
hence $p,q>0$. So the latter two cases in the above are impossible to appear, and (1) is
shown.

To prove (2), we first show

\begin{lemma}\label{L-5.3}
Let $n\in\bN$ and assume that $p,q\ne0$ and $s>0$. If the function \eqref{F-5.8} is
jointly convex, then
$$
\biggl({A^{1/q}+B^{1/q}\over2}\biggr)^q
\le\biggl({A^{-1/p}+B^{-1/p}\over2}\biggr)^{-p}
$$
for every $A,B\in\bP_n$.
\end{lemma}

\begin{proof}
The assumption means that
\begin{align*}
&\bigg\|\biggl\{\biggl({A_1+A_2\over2}\biggr)^{p/2}\biggl({B_1+B_2\over2}\biggr)^q
\biggl({A_1+A_2\over2}\biggr)^{p/2}\biggr\}^s\bigg\|_\infty \\
&\qquad\le{\big\|(A_1^{p/2}B_1^qA_1^{p/2})^s\big\|_\infty
+\big\|(A_2^{p/2}B_2^qA_2^{p/2})^s\big\|_\infty\over2}
\end{align*}
for every $A_i,B_i\in\bP_n$, $i=1,2$. Since $s>0$, the above inequality implies that
if $A_i^{p/2}B_i^qA_i^{p/2}\le I$ for $i=1,2$ then
$$
\biggl({A_1+A_2\over2}\biggr)^{p/2}\biggl({B_1+B_2\over2}\biggr)^q
\biggl({A_1+A_2\over2}\biggr)^{p/2}\le I,
$$
that is, if $B_i^q\le A_i^{-p}$ for $i=1,2$ then
$((B_1+B_2)/2)^q\le((A_1+A_2)/2)^{-p}$. In particular, letting $A_i=B_i^{-q/p}$ gives
$$
\biggl({B_1+B_2\over2}\biggr)^q\le\biggl({B_1^{-q/p}+B_2^{-q/p}\over2}\biggr)^{-p},
$$
which is clearly equivalent to the desired inequality.
\end{proof}

The ``if\," part of the next lemma is rather easy as given in \cite{MPP} (also
\cite[Chapter 4]{FMPS}) in a more general form. However, it would be beyond the scope of
this paper if we supply counterexamples to prove the ``only if\," part. So we leave the
details of the proof to a separate paper \cite{AuHi}. The lemma includes the
cases $p=0$ or $q=0$ for completeness while $p,q\ne0$ in Theorem \ref{T-5.2}.

\begin{lemma}\label{L-5.4}
For $p,q\in\bR$ consider the matrix inequality
\begin{equation}\label{F-5.9}
\biggl({A^p+B^p\over2}\biggr)^{1/p}\le\biggl({A^q+B^q\over2}\biggr)^{1/q},
\end{equation}
where $((A^p+B^p)/2)^{1/p}$ for $p=0$ means
$$
\lim_{p\to0}\biggl({A^p+B^p\over2}\biggr)^{1/p}
=\exp\biggl({\log A+\log B\over2}\biggr).
$$
Then inequality \eqref{F-5.9} holds for every $n\in\bN$ and every
$A,B\in\bP_n$ (or equivalently, for every $A,B\in\bP_2$ with fixed $n=2$) if
and only if one of the following is satisfied:
$$
\begin{cases}
p=q, \\
1\le p<q, \\
p<q\le-1, \\
p\le-1,\ q\ge1, \\
1/2\le p<1\le q, \\
p\le-1<q\le-1/2.
\end{cases}
$$
\end{lemma}

\noindent
{\it Proof of Theorem \ref{T-5.2} (continued).}\enspace
Let $s>0$ and assume joint convexity of \eqref{F-5.8} when $n=2$. As in
the proof of Proposition \ref{P-4.1}\,(1) (just replace $\Tr$ with $\|\cdot\|_\infty$ and
concavity with convexity) we see that $(a^p+b^p)^s$ is convex in $a,b>0$. By \eqref{F-4.1}
we have either $ps\ge1$ and $p\ge1$, or $ps<0$ and $p\le1$. Similarly, we have either
$qs\ge1$ and $q\ge1$, or $qs<0$ and $q\le1$. Therefore, $p,q\in(-\infty,0)\cup[1,\infty)$.
Furthermore, since the Hessian of $x^{-q}y^q$ is $-q^2x^{-2q-2}y^{2q-2}<0$,
$x^{-q}y^q$ cannot be jointly convex in $x,y>0$ so that the case $p=-q$ is excluded.
Thus, by Lemmas \ref{L-5.3} and \ref{L-5.4}, one of the following must be satisfied:
\begin{itemize}
\item[(a)] $1\le1/q<-1/p$,
\item[(b)] $1/q<-1/p\le-1$,
\item[(c)] $1/q\le-1$, $-1/p\ge1$,
\item[(d)] $1/2\le1/q<1\le-1/p$,
\item[(e)] $1/q\le-1<-1/p\le-1/2$.
\end{itemize}
When $p,q<0$ and so $1/q<0<-1/p$, (c) must hold so that  $-1\le p,q<0$, which is the
first case of \eqref{F-5.2}. When $p<0$ and $q\ge1$ and so $0<1/q\le1$ and $-1/p>0$, (d)
or (a) with $1/q=1$ must hold so that $1<q\le2$ and $-1\le p<0$, or $q=1$
and $-1<p<0$. Moreover, since $x^{(p+q)s}$ is convex in $x>0$, $(p+q)s\ge1$. So we have
the second case of \eqref{F-5.2}. When $p\ge1$ and $q<0$, we similarly have the third
case of \eqref{F-5.2}. Finally, the case where $p\ge1$ and $q\ge1$ cannot be compatible
with any of (a)--(e), so this case does not appear.
\end{proof}

\section*{Concluding remarks}
After completing this paper we have obtained essential improvements on Theorem \ref{T-1.1}
as follows: Let $f$ be a real function on $(0,\infty)$, $0\le p,q\le1$ with
$(p,q)\ne(0,0)$, and $\Phi,\Psi$ be as in Theorem \ref{T-1.1}. If either
\begin{itemize}
\item[(a)] $f(x^{p+q})$ is operator monotone on $(0,\infty)$, or
\item[(b)] $f$ is non-decreasing and $f(x^2)$ is concave on $(0,\infty)$,
\end{itemize}
then the functions
\begin{align*}
(A,B)\in\bP_n\times\bP_m&\longmapsto
\Tr f\bigl(\Phi(A^p)^{1/2}\Psi(B^q)\Phi(A^p)^{1/2}\bigr) \\
&\qquad\mbox{and}\quad
\Tr f\Bigl(\bigl\{\Phi(A^{-p})^{1/2}\Psi(B^{-q})\Phi(A^{-p})^{1/2}\bigr\}^{-1}\Bigr)
\end{align*}
are jointly concave.

Applying the case (a) to $f(x)=x^s$ with $0\le s\le1/(p+q)$ (also the case (b) to
$f(x)=x^s$ with $0\le s\le1/2$) we conclude that the function \eqref{F-1.1} is jointly
concave if either $0\le p,q\le1$ and $0\le s\le1/(p+q)$, or $-1\le p,q\le0$ and
$1/(p+q)\le s\le0$. In view of Proposition \ref{P-4.1}\,(2), this completely settles the
joint concavity question for \eqref{F-1.1}. Applying (b) to $f(x)=-x^s$ with $s\le0$ we
also see that \eqref{F-1.1} is jointly convex if either $0\le p,q\le1$ and $s\le0$, or
$-1\le p,q\le0$ and $s\ge0$. This considerably improves Theorem \ref{T-1.1}\,(2) though
there is still a gap from a necessary condition in Proposition \ref{P-4.4}\,(2).

The proof for the case (a) is an adaptation of Epstein's method in the proof of Theorem
\ref{T-1.1} based on the integral expression of an operator monotone function $f$.
For the case (b) we first prove by means of matrix differential calculus that
$\Tr f(A^{1/2}BAB^{1/2})$ is jointly concave if and only if $f$ satisfies condition (b).
Then the result follows from an argument as in Remark \ref{R-4.3}. The details of these
and related matters will be presented elsewhere.

\section*{Acknowledgments}

The author would like to thank Jean-Christophe Bourin who suggested him how to extend
concavity results for trace functions to symmetric anti-norm functions, that is the main
idea of this paper. He acknowledges support by Grant-in-Aid for Scientific Research
(C)21540208.

\end{document}